\documentclass[a4paper]{amsart}

\usepackage{amssymb}
\usepackage{stmaryrd}
\swapnumbers
\makeatletter
\def\section{\@startsection{section}{1}%
  \z@{.7\linespacing\@plus\linespacing}{.5\linespacing}%
  {\normalfont\bfseries\centering}}
\def\@secnumfont{\bfseries}
\makeatother

\renewcommand{\o}{\circ}
\def\frak{\mathfrak}
\def\Bbb{\mathbb}
\def\Cal{\mathcal}
\def\sideremark#1{\ifvmode\leavevmode\fi\vadjust{\vbox to0pt{\vss%
  \hbox to 0pt{\hskip\hsize\hskip1em%
  \vbox{\hsize3cm\tiny\raggedright\pretolerance10000%
  \noindent #1\hfill}\hss}\vbox to8pt{\vfil}\vss}}}%

\newcommand{\al}{\alpha}
\newcommand{\be}{\beta}
\newcommand{\ga}{\gamma}

\newcommand{\ka}{\kappa}
\newcommand{\la}{\lambda}

\newcommand{\om}{\omega}
\newcommand{\ph}{\varphi}

\newcommand{\ze}{\zeta}

\newcommand{\La}{\Lambda}

\newcommand{\Ps}{\Psi}

\newcommand{\Om}{\Omega}

\newcommand{\ups}{\upsilon}

\newcommand{\Ad}{\operatorname{Ad}}

\renewcommand{\exp}{\operatorname{exp}}
\newcommand{\id}{\operatorname{id}}

\newcommand{\tr}{\operatorname{tr}}

\newcommand{\End}{\operatorname{End}}

\renewcommand{\o}{\circ}

\newcommand{\pmat}[1]{\begin{pmatrix}#1\end{pmatrix}}

\renewcommand{\P}{\operatorname{\Cal P}}
\let\del=\partial
\let\[=\llbracket
\let\]=\rrbracket
\let\x=\times

\def\g{\frak g}

\def\p{\frak p}
\def\q{\frak q}

\def\X{\frak X}
\def\({\big(}
\def\){\big)}
\def\R{\Bbb R}

\def\H{\Bbb H}
\def\I{\Bbb I}
\def\G{{\Cal G}}
\def\L{{\Cal L}}
\def\tG{\tilde G}
\def\tP{\tilde P}

\def\tg{\tilde\g}
\def\tp{\tilde\p}

\def\tom{\tilde\om}
\def\ddt#1{\tfrac{d}{dt}\ifx#1\nic\else\big\vert_{#1}\fi}
\def\.{\hbox to5pt{\hss$\cdot$\hss}}
\def\span#1{\langle#1\rangle}

\newtheorem*{prop*}{Proposition}
\newtheorem{thm}[subsection]{Theorem}
\newtheorem*{thm*}{Theorem}

\newtheorem*{lem*}{Lemma}

\newtheorem*{cor*}{Corollary}
\newtheorem*{def*}{Definition}

\begin{document}

\title{Lie contact structures and chains\\} 
\author{Vojt\v ech \v Z\'adn\'ik}
\begin{abstract}
  Lie contact structures generalize the classical Lie sphere geometry of
  oriented hyperspheres in the standard sphere.
  They can be equivalently described as parabolic geometries corresponding
  to the contact grading of orthogonal real Lie algebra.
  It follows the underlying geometric structure can be interpreted in
  several equivalent ways.
  In particular, we show this is given by a split-quaternionic structure on
  the contact distribution, which is compatible with the Levi bracket.

  In this vein, we study the geometry of chains,  a distinguished family of 
  curves appearing in any parabolic contact geometry.
  Also to the system of chains there is associated a canonical parabolic
  geometry of specific type.
  Up to some exceptions in low dimensions, 
  it turns out this can be obtained by an extension of the parabolic geometry
  associated to the Lie contact structure if and only if the latter is locally
  flat.
  In that case we can show that chains are never geodesics of an affine
  connection, hence, in particular, the path geometry of chains is always
  non-trivial.
  Using appropriately this fact, we conclude that the path geometry of chains 
  allows to recover the Lie contact structure, hence, in particular, 
  transformations preserving chains must preserve the Lie contact structure.
\end{abstract}
\address{Masaryk University, Brno, Czech Republic}
\email{zadnik@math.muni.cz}

\maketitle


\section{Introduction}			\label{1}
The Lie sphere geometry is the geometry of oriented hyperspheres
in the standard sphere established  by S.~Lie.
A generalization of the corresponding geometric structure to general smooth 
manifold is provided by \cite{SY} and further studied by other authors.
Here we briefly present the basic ideas and outline the purposes of this
paper.

\subsection{Classics}
Let us consider the vector space $\R^{n+4}$ with an inner product of signature
$(n+2,2)$.
The projectivization of the cone of non-zero null-vectors in $\R^{n+4}$ is  a
hyperquadric in the projective space $\R\Bbb P^{n+3}$, 
which is called the \textit{Lie quadric} and denoted by $Q^{n+2}$.
The standard sphere $S^{n+1}$ is then realized as the intersection of $Q^{n+2}$ 
with a hyperplane in $\R\Bbb P^{n+3}$.
There is a bijective correspondence between the Lie quadric $Q^{n+2}$ 
and the set of the so called kugels of $S^{n+1}$.
The kugel of $S^{n+1}$ is an oriented hypersphere or a point 
(called a point sphere) in $S^{n+1}$.
The point is that kugels corresponding to the same projective line are in
oriented contact, with the point sphere as the common contact point.
Hence each projective line in $Q^{n+2}$ is uniquely represented by the common 
contact point and the common unit normal of the family of kugels in contact.
This establishes a bijective correspondence between the set of projective
lines in $Q^{n+2}$ (i.e. isotropic planes in $\R^{n+4}$) and the unit tangent 
sphere bundle of $S^{n+1}$, denoted by $T_1(S^{n+1})$.
Either of the two is then understood as the model Lie contact structure in
dimension $2n+1$.

By definition, the Lie transformation group $G$ is the group of
projective transformations of $\R\Bbb P^{n+3}$ preserving the Lie quadric.
Hence $G$ is isomorphic to $PO(n+2,2)$, the quotient of $O(n+2,2)$ by its 
center which is $\{\pm\id\}$.
The group $G$ acts transitively (and effectively) on the set of projective
lines in $Q^{n+2}$, and hence on $T_1(S^{n+1})$, and preserves the canonical
contact structure.
As a homogeneous space, $T_1(S^{n+1})\cong G/P$ where $P\subset G$ is the 
stabilizer of some element.
It turns out $P$ is a parabolic subgroup of $G$.

The generalization of the concepts above to general contact manifold
yields the notion of the Lie contact structure, which is defined in section 
3 in \cite{SY} as a reduction of the adapted frame bundle to an appropriate
subgroup of the structure group. 
Applying the Tanaka's theory, the equivalence problem for Lie contact
manifolds is solved in principle by Theorem 4.3 in \cite{SY} and rather 
explicitly in \cite{Miy}.
Anyway, there is established an equivalence between Lie contact structures and
normal Cartan geometries of type $(G,P)$.
Important examples of Lie contact structures are observed on the unit tangent 
sphere bundles of Riemannian manifolds.

\subsection{General signature}
The ideas above allow a natural generalization considering the inner 
product on $\R^{n+4}$ to have an arbitrary signature $(p+2,q+2)$, where $p+q=n$.
Following the previous approach, we still consider the space of projective lines
in the Lie quadric $Q^{n+2}$, i.e.\ the space of isotropic planes in
$\R^{n+4}$, as the model.
As a homogeneous space, this is isomorphic to $G/P$, where $G=PO(p+2,q+2)$ and
$P\subset G$ is the stabilizer of an isotropic plane in $\R^{n+4}$.
It is then natural to define the \textit{Lie contact structure of signature $(p,q)$}
as the underlying structure of a parabolic geometry of type $(G,P)$ which is
the meaning of the definition in \ref{2.2}, which we adopt from section 4.2.5 in
\cite{CS}.
Note that in low dimensions the Lie contact structure may have a specific
flavour, which is briefly discussed in remark \ref{rem1}.

First of all, 
the Lie contact structure on $M$ involves a contact structure $H\subset TM$
so that $H\cong L^*\otimes R$, where $L$ and $R$ are auxiliary vector bundles 
over $M$ of rank 2 and $n$.
Note that the auxiliary bundles has almost no intrinsic geometrical meaning,
however, for $H\cong L^*\otimes R$, there is a distinguished subset in each
$H_x$ consisting of all the elements of rank one.
This is the \textit{Segre cone} which plays a role in the sequel.
The maximal linear subspaces contained in the cone have dimension $n$ and it turns out 
they are isotropic with respect to the Levi bracket.
Characterization of the Lie contact structure in these terms is provided by
Proposition \ref{prop1}.

On the other hand, the tensor product structure of $H_x$ can be naturally 
rephrased as a  \textit{split-quaternionic structure}. 
The compatibility with the Levi bracket is easy to express and
it turns out  these data also characterize the Lie contact structure entirely, 
Proposition \ref{prop2}.
This is the convenient interpretation of the Lie contact structure we further
use below.
(Note that in both cases there is minor additional input involved, namely a
fixed trivialization of a line bundle over $M$.)

\subsection{Chains}
As for any parabolic contact structure, there is a general concept of
\textit{chains} which form a distinguished family of  curves generalizing 
the Chern--Moser chains on CR manifolds of hypersurface type.
As unparametrized curves (paths), chains are uniquely determined by a tangent
direction, transverse to the contact distribution, in one point.
A path geometry of chains can be equivalently described as a regular normal
parabolic geometry of a specific type over the open subset of $\P TM$
consisting of all non-contact directions in $TM$.
It turns out it is possible to relate the parabolic contact geometry and the
path geometry of chains on the level of Cartan geometries directly, i.e.\
without the prolongation.
This was done for Lagrangean contact structures  and CR structures of
hypersurface type in \cite{CZ}. 

In the rest of present paper we follow the mentioned construction and the consequences 
for Lie contact structures.
First, dealing with the homogeneous model, one explicitly describes the data allowing 
the direct relation between the Lie contact geometry and the path geometry of 
chains, section \ref{3.4}.
Second, the construction in general is compatible with the normality condition
if and only if  the parabolic contact structure is torsion free.
For Lie contact structures in dimension grater than or equal
to 7, the normality and torsion freeness imply the structure is locally flat,
see Theorem \ref{th2}.

Comparing to earlier studies in \cite{CZ,CZ2}, this brings a rather strong restriction.
Under this assumption the tools we use may seem a bit non-proportional, however, 
to our knowledge there is no elementary argument covering the results below,  
not even in the homogeneous model.
Analyzing the curvature of the induced Cartan geometry, we conclude with some 
applications for locally flat Lie contact structures:
Chains are never geodesics of an affine connection, hence, in particular, the
path geometry of chains is never flat, Theorem \ref{th3}.
Therefore the harmonic curvature of the path geometry of chains (and an
efficient interpretation of the structure) allows one to recover the 
Lie contact structure. 
Hence, contact diffeomorphisms preserving chains must preserve the Lie 
contact structure, Theorem \ref{th4}.

\subsection*{Acknowledgements}
Author thanks to Stuart Armstrong, Jan Slov\'ak, and Andreas \v Cap in
particular for  number of valuable discussions.
Support by the grant 201/06/P379 of the Grant Agency of Czech Republic is acknowledged too.

\section{Lie contact structures}	\label{2}
In this section we bring the general definition of Lie contact structures with
the alternative interpretations of the underlying geometric structure.
We start with a necessary background.

\subsection{Parabolic contact structures}		\label{2.0}
For a semisimple Lie group $G$ and a parabolic subgroup
$P\subset G$, \textit{parabolic geometry} of type $(G,P)$ on a smooth manifold 
$M$ consists of a principal $P$-bundle $\G\to M$ and a Cartan connection 
$\om\in\Om^1(\G,\g)$, where $\g$ is the Lie algebra of $G$.
The Lie algebra $\g$ is always equipped with the grading 
of the form $\g=\g_{-k}\oplus \dots \oplus \g_{0} \oplus \dots \oplus\g_{k}$ 
such that the Lie algebra $\p$ of $P$ is $\p = \g_0 \oplus \dots \oplus \g_k$. 
By $G_0$ we denote the subgroup in $P$, with the Lie algebra $\g_0$,
consisting of  all elements in $P$ whose adjoint action preserves the grading of $\g$. 
The grading of $\g$ induces a $P$-invariant filtration of $\g$, which
gives rise to a filtration of the tangent bundle $TM$, due to the usual identification 
$TM\cong\G\x_P\g/\p$ via the Cartan connection $\om$.
On the graded vector bundle corresponding to this filtration, there is an
algebraic bracket induced by the Lie bracket of vector fields, which is called
the \textit{Levi bracket}.
The parabolic geometry is \textit{regular} if the Levi bracket corresponds to
the bracket in $\g$ under the identification above.

\textit{Parabolic contact geometry} is a parabolic geometry whose underlying 
geometric structure consists of a contact distribution $H\subset TM$ and some 
additional structure on $H$. 
These correspond to contact gradings of simple Lie algebras as follows.
The \textit{contact grading} of a simple Lie algebra is a grading 
$\g=\g_{-2}\oplus\g_{-1}\oplus\g_0\oplus\g_1\oplus\g_2$ such
that $\g_{-2}$ is one dimensional and the Lie bracket $[\ ,\ ]:\g_{-1}\x\g_{-1}
\to\g_{-2}$ is non-degenerate.
Let us consider regular parabolic geometry of type $(G,P)$ such that the Lie 
algebra of $G$ admits a contact grading.
Then the corresponding filtration of the tangent bundle of $M$ is just a 
distribution $H\subset TM$, which turns out to be contact, the Levi bracket 
$\L:H\x H\to TM/H$ is non-degenerate, and the reduction of 
$\operatorname{gr}(TM):=(TM/H)\oplus H$ to the structure group $G_0$ gives
rise to the additional structure on $H$.

For general parabolic geometry $(\G\to M,\om)$, the curvature is often described by
the so called \textit{curvature function} $\ka:\G\to\La^2(\g/\p)^*\otimes\g$,
which is given by 
\begin{equation*}
  \ka(u)(X+\p,Y+\p)=d\om(\om^{-1}(X)(u),\om^{-1}(Y)(u))+[X,Y].
\end{equation*}
The Killing form on $\g$ provides an identification $(\g/\p)^*$ with $\p_+$,
hence the curvature function is viewed as having values in
$\La^2\p_+\otimes\g$.
The grading of $\g$ induces grading also to this space, which brings the notion of
homogeneity.
In particular, parabolic geometry is regular if and only if  the curvature function 
has values in the part of positive homogeneity.
Parabolic geometry is called \textit{torsion free} if $\ka$ has values in 
$\La^2\p_+\otimes \p$; note that torsion free parabolic geometry is
automatically regular.
Next, parabolic geometry is called \textit{normal} if $\partial^* \circ \ka =0$, where
$\partial^*:\La^2\p_+\otimes\g\to\p_+\otimes\g$ is the differential in
the standard complex computing the homology $H_*(\p_+,\g)$ of $\p_+$ with 
coefficients in $\g$.
For a regular normal parabolic geometry, the \textit{harmonic curvature} 
$\ka_H$  is the composition of $\ka$ with the natural projection 
$\ker(\partial^*)\to H_2(\p_+,\g)$.
By definition, $\ka_H$ is a section of $\G\x_P H_2(\p_+,\g)$ and the point is
it can be interpreted in terms of the underlying structure.
For all details on parabolic geometries we primarily refer to \cite{CS}.

\subsection{Contact grading of $\frak{so}(p+2,q+2)$}		\label{2.1}
Consider the inner product on $\R^{n+4}$ given by the matrix 
$$
\pmat{0&0&-\I_2\\ 0&\I_{p,q}&0\\ -\I_2&0&0},
$$
where $\I_{p,q}=\pmat{\I_p&0\\0&-\I_q}$ and $\I_r$ is the unit matrix of rank
$r$.
According to this choice, the Lie algebra $\g=\frak{so}(p+2,q+2)$ has got the
following form  with blocks of sizes 2, $n$, and 2:
\begin{equation*}
  \pmat{A&U&w\Bbb J\\ X&D&\I_{p,q}U^t\\ z\Bbb J&X^t\I_{p,q}&-A^t},
\end{equation*}
with $\Bbb J:=\pmat{0&1\\-1&0}$, where $z,w\in\R$, $D\in\frak{so}(p,q)$, and
$X,A, U$ are real matrices of size $n\x 2, 2\x 2, 2\x n$, respectively.
The contact grading of $\g$ is read along the diagonals so that $z$
parametrizes $\g_{-2}$, $X$ corresponds to $\g_{-1}$, the pair $(A,D)$ to 
$\g_0$, $U$ to $\g_1$, and $w$ to $\g_2$.
In particular, $X\in\g_{-1}$ is understood as an element of 
$\R^{2*}\otimes\R^n$, the space of linear maps from $\R^2$ to $\R^n$.

If we write $X\in\g_{-1}$ as the matrix $(X_1,X_2)$ with columns
$X_1,X_2\in\R^n$, and similarly for $Y$, then  the Lie bracket 
$[\ ,\ ]:\g_{-1}\x\g_{-1}\to\g_{-2}$ is explicitly given by
\begin{equation} 		\label{eq1}
  [X,Y]=(\span{X_1,Y_2}-\span{X_2,Y_1})\.e,
\end{equation}
where $\span{\ ,\ }$ denotes the inner product on $\R^n$ corresponding to
$\I_{p,q}$ and $e$ is the generator of $\g_{-2}$ corresponding to $z=1$ in the
description above.
Indeed, this bracket is non-degenerate.
Since $\g_{-1}$ is identified with the space of linear maps from $\R^2$ to $\R^n$, 
an endomorphism of $\R^m$ acts also on $\g_{-1}$ by the composition. 
In matrices, this is given by the left multiplication $X\mapsto CX$.
From \eqref{eq1} it is obvious the bracket $\g_{-1}\x\g_{-1}\to\g_{-2}$ is 
invariant under the action of endomorphisms on $\R^n$ preserving the inner product,
i.e.\ 
\begin{equation} 		\label{eq2}
  [CX,CY]=[X,Y]
\end{equation}
for any $C\in O(p,q)$ and any $X,Y\in\g_{-1}$.
Similarly, an endomorphism $A$ of $\R^2$ acts on $\g_{-1}$ from the right by
$X\mapsto XA$. 
An easy direct calculation shows that this is compatible with the bracket so
that 
\begin{equation} 		\label{eq3}
  [XA,YA]=\det A\.[X,Y]
\end{equation}
for any $A\in\frak{gl}(2,\R)$ and any $X,Y\in\g_{-1}$.

Following the introductory section, let the group to the Lie algebra 
$\g=\frak{so}(p+2,q+2)$ be $G:=PO(p+2,q+2)$.
The parabolic subgroup $P\subset G$ with the Lie algebra $\p=\g_0\oplus\g_1\oplus\g_2$
is represented by block upper triangular matrices in $G=PO(p+2,q+2)$ 
with blocks of sizes 2, $n$, and 2.
Obviously, $P$ stabilizes the plane spanned by the first two vectors of the
standard basis in $\R^{n+4}$ (which is indeed isotropic).
The subgroup $G_0\subset P$ is represented by block diagonal matrices in $P$ with
blocks of the same size.
Explicitly, $G_0$ is formed by the classes of matrices of the form 
\begin{equation*}
  \pmat{B&0&0\\ 0&C&0\\ 0&0&(B^{-1})^t},
\end{equation*}
where $B\in GL(2,\R)$ and $C\in O(p,q)$.
Each class has just two elements which differ by the sign, hence
$G_0\cong (GL(2,\R)\x O(p,q))/\{\pm\id\}$.
Let us represent an element of $G_0$ by the pair $(B,C)$ and an element of
$\g_-$ by $(z,X)$ as above.
Then a direct computation shows the adjoint action of $G_0$ on $\g_-$
is explicitly given by 
$$ 
  \Ad(B,C)(z,X)=(\det B^{-1}\.z,CXB^{-1}),
$$
which indeed does not depend on the representative matrix.
This shows the restriction to $\g_{-1}\cong\R^{2*}\otimes\R^n$ comes from the
product of the standard representation of $GL(2,\R)$ on $R^2$ and 
of $O(p,q)$ on $\R^n$.
Combining with \eqref{eq2} and \eqref{eq3}, the action of $G_0$ on $\g_-$ is 
compatible with the Lie bracket $\g_{-1}\x\g_{-1}\to\g_{-2}$, i.e.\ the
bracket is indeed $G_0$-equivariant.

\subsection{Lie contact structures}		\label{2.2}
  \textit{Lie contact structure of signature $(p,q)$} on a smooth manifold $M$ 
  of dimension $2n+1$, $n=p+q$, consists of the following data:
  \begin{itemize}
  \item a contact distribution $H\subset TM$,
  \item two auxiliary vector bundles $L\to M$ and $R\to M$ of rank 2 and $n$,
  \item a bundle metric of signature $(p,q)$ on $R$, where $p+q=n$,
  \item an isomorphism $H\cong L^*\otimes R$,
  \end{itemize}
  such that the Levi bracket is invariant under the action of $O(R_x)\cong
  O(p,q)$ on $H_x$, for each $x\in M$, i.e.\
  $\L(\xi,\eta)=\L(\ga\o\xi,\ga\o\eta)$ for any $\ga\in O(R_x)$ and any
  $\xi,\eta\in L_x^*\otimes R_x$.

As we announced above, the Lie contact structure should coincide with 
the underlying structure corresponding to the parabolic geometry of type
$(G,P)$, i.e. with the parabolic contact structure corresponding to the
contact grading of $\g=\frak{so}(p+2,q+2)$ . 
This is really the case and the equivalence is formulated as Proposition 4.2.5 
in \cite{CS}:
\begin{thm*}				\label{th0}
  Let $G=PO(p+2,q+2)$ and $P\subset G$ be the stabilizer of an isotropic plane.
  Then the category of regular normal parabolic geometries of type $(G,P)$
  is equivalent to the category of Lie contact structures of signature $(p,q)$.
\end{thm*}

\subsection*{Remarks}			\label{rem1}
(1) Note that the section 4.2.5 in \cite{CS} develops according to the
choice $G=O(p+2,q+2)$.
However, different choices of the Lie group to the given Lie algebra coincide
generally up to a (usually finite) cover.
In particular, the freedom in the choice does not affect the local description 
of the underlying structure.

(2) Although we define the Lie contact structure for general $n=p+q$, there
are some specific features in low dimensions. 
For instance, the Lie contact structure of signature $(2,0)$ (on
5-dimensional manifold) is basically equivalent to a non-degenerate CR
structure with indefinite Levi form, which is discussed e.g.\ in 
section 5 in \cite{SY} in some detail.
In our terms, the equivalence is provided by the Lie algebra isomorphism
$\frak{so}(4,2)\cong\frak{su}(2,2)$ and the uniqueness of the contact grading.
Similarly, for $\frak{so}(3,3)\cong\frak{sl}(4,\R)$, the Lie contact structure
of signature $(1,1)$ is equivalent to the Lagrangean contact structure (on a
5-manifold).
The interpretation of $\frak{so}(3,2)\cong\frak{sp}(4)$ is a bit different, 
since the Lie contact structure of signature $(1,0)$ is rather trivial and 
usually excluded from the considerations.
However, the structure corresponding to the contact grading of $\frak{sp}(4)$
is the contact projective structure (on a 3-manifold).
If dimension of the base manifold is grater then or equal to 7, the Lie contact structure
starts to work with no specific issue, up to the only exception in dimension
9, which is due to $\frak{so}(6,2)\cong\frak{so}^*(8)$.
An interested reader may investigate the equivalent structure  behind that isomorphism.

\subsection{Segre cone}			\label{2.3}
For $H_x\cong L^*_x\otimes R_x$, there is a distinguished subset 
$\Cal C_x\subset H_x$, the so called \textit{Segre cone}, consisting of all the 
linear maps $L_x\to R_x$ of rank one. 
On the level of the Lie algebra $\g=\frak{so}(p+2,q+2)$, the cone $\Cal C_x\subset
H_x$ corresponds to the $G_0$-invariant subset of $\g_{-1}\cong\R^{2*}\otimes\R^n$
consisting of the elements of the from $f\otimes u$ for some $f\in R^{2*}$ and 
$u\in\R^n$. 
It follows that $f\otimes\R^n$ is a maximal linear subspace in the cone (consisting of the 
linear maps $\R^2\to\R^n$ with the common kernel $\ker f$).
For arbitrary elements $X=f\otimes u_1$ and $Y=f\otimes u_2$ from $f\otimes\R^n$, 
the substitution into \eqref{eq1} yields that $[X,Y]=0$, i.e.\ the subspace 
$f\otimes\R^n\subset\g_{-1}$ is isotropic with respect to the 
Lie bracket $\g_{-1}\x\g_{-1}\to\g_{-2}$.
Hence for general rank-one elements $f_1\otimes u_1,f_2\otimes u_2$ from 
$\R^{2*}\otimes\R^n=\g_{-1}$, it follows that 
$$
  [f_1\otimes u_1,f_2\otimes u_2]=|f_1,f_2|\span{u_1,u_2}\.e,
$$
where $|\ ,\ |$ denotes the standard exterior product on $\R^{2*}$, $\span{\ ,\ }$ is
the standard inner product on $\R^n$ of signature $(p,q)$, and $e$ is the generator of 
$\g_{-2}$ as in \ref{2.1}.
This shows the Lie bracket is actually of the form 
$\La^2\R^2\otimes S^2\R^{n*}\otimes\g_{-2}$ and it determines the inner
product on $\R^n$ provided we fix an identification $\La^2\R^2\otimes\g_{-2}\cong\R$.
This is provided by the $G_0$-invariant mapping  $f_1\wedge f_2\otimes
ze\mapsto z|f_1,f_2|$, which also yields a trivialization of the line bundle 
$\La^2L\otimes TM/H$ over any Lie contact manifold $M$.
\begin{prop*}			\label{prop1}
  A Lie contact structure on a smooth manifold $M$ of dimension $2n+1$ 
  is equivalent to the following data:
  \begin{itemize}
  \item a contact distribution $H\subset TM$,
  \item two auxiliary vector bundles $L\to M$ and $R\to M$ of rank 2 and $n$,
  \item an isomorphism $H\cong L^*\otimes R$,
  \item an isomorphism $\La^2L^*\cong TM/H$,
  \end{itemize}
  such that for $\ph\in L_x^*$ the subspace $\ph\otimes R_x\subset H_x$ is
  isotropic with respect to the Levi bracket, i.e.\
  $\L(\ph\otimes\ups_1,\ph\otimes\ups_2)=0$ for any
  $\ups_1,\ups_2\in R_x$.
\end{prop*}
\begin{proof}
  According to the discussion above, we only need to construct a Lie contact
  structure on $M$ from the later data.
  In general, the Levi bracket is a section of the bundle $\La^2 H^*\otimes TM/H$, 
  which  decomposes   according to the isomorphism $H\cong L^*\otimes R$ as 
  $(\La^2L\otimes S^2R^*\otimes TM/H)\oplus(S^2L\otimes\La^2R^*\otimes TM/H)$.
  Since we assume $\ph\otimes R$ is isotropic with respect to $\L$, for any $\ph$,
  the Levi bracket $\L$ factorizes through the first summand only.
  Since we assume a fixed identification of line bundles $\La^2L^*$ and
  $TM/H$, i.e. a  trivialization of $\La^2L\otimes TM/H$, the Levi
  bracket determines a (non-degenerate) bundle  metric on $R$, which accomplishes 
  the Lie contact structure on $M$. 
  In particular, if $O(R_x)$ is the group of orthogonal transformations of
  $R_x$ according to this inner product, then the Levi bracket is  by construction
  invariant under the action of $O(R_x)$, over each $x\in M$.
\end{proof}

\subsection{Split quaternions}			\label{2.4}
The algebra of \textit{split quaternions} $\H_s$ is the four-dimensional real algebra
admitting a basis $(1,i,j,k)$ such that $i^2=j^2=1$ and $k=ij=-ji$ (consequently
$k^2=-1$, $kj=-jk=i$, etc.).
As an associative algebra, $\H_s$ is isomorphic to the space of $2\x2$ real
matrices so that the norm on $\H_s$ corresponds to the determinant.
This correspondence is explicitly  given by
\begin{equation}			\label{eq5}
  1\mapsto\pmat{1&0\\0&1},\ i\mapsto\pmat{1&0\\0&-1},\ 
  j\mapsto\pmat{0&1\\1&0},\ k\mapsto\pmat{0&1\\-1&0}.
\end{equation}
In particular, the norm squared of $i,j$, and $k$, equals $-1,-1$, and 1,
respectively.
For general imaginary quaternion $q=ai+bj+ck$, the norm squared is
$|q|^2=-a^2-b^2+c^2$, which  can also be given by $q^2=-|q|^2$.

The \textit{split-quaternionic structure} on a vector bundle $H$ over $M$
is defined as a 3-dimensional subbundle $\Cal Q\subset\End H$ admiting a basis
$(I,J,K)$ such that $I^2=J^2=-\id$ and $K=I\o J=-J\o I$.
Note that, in contrast to the true quaternionic structures, 
split-quaternionic structure may exist on vector budles of any even rank. 

\begin{lem*}			\label{lem0}
  Let $H\to M$ be a vector bundle of rank $2n$. 

  (1) A split-quaternionic structure $\Cal Q$ on $H$ is equivalent to an isomorphism
  $H\cong L^*\otimes R$, where $L\to M$ and $R\to M$ are vector bundles of
  rank 2 and $n$, respectively. 

  (2) In particular, a non-zero element $\xi$ in $H_x\cong L^*_x\otimes R_x$ has rank one 
  if and only if  there is $A\in\Cal Q_x$ such that $A(\xi)=\xi$ (consequently, 
  $A$ is a product structure). 

  (3) Moreover, the maximal linear subspaces contained in the cone of
  rank-one elements in $H_x$ are the $+1$-eigenspaces of the product 
  structures in $\Cal Q_x$.
\end{lem*}

\begin{proof}
  (1) Let $H\cong L^*\otimes R$.
  Since $L_x\cong\R^2$  and $R_x\cong\R^n$, for each $x\in M$, the natural
  action of $\H_s$ on $\R^2$ extends, by the composition, to the action on 
  $\R^{2*}\otimes\R^n$.
  Denoting by $I,J$, and $K$ the endomorphisms of $H_x\cong\R^{2*}\otimes\R^n$
  corresponding to the elements $i,j$, and $k$ from \eqref{eq5}, we obtain a
  split-quaternionic structure $\Cal Q_x:=\span{I,J,K}$.

  Conversely, let a split-quaternionic structure $\Cal Q=\span{I,J,K}$ be given. 
  For each $x\in M$, the subspace $H_x\subset T_xM$ decomposes to the eigenspaces 
  $H_x^+\oplus H_x^-$ with respect to the product structure $I$.
  Let us define the auxiliary vector bundles over $M$ to be $R:=H^+\subset H$ and
  $L:=\span{I,J}^*$, the dual of $\span{I,J}\subset\Cal Q$.
  (Obviously, $R$ and $L$ have got rank $n$ and 2, respectively.)
  We claim that $L^*\otimes R\cong H$ under the mapping $\ph\otimes X\mapsto
  \ph(X)$:
  By definition, $\ph\in L^*$ is a linear combination of the endomorphisms $I,J$ 
  and the restriction of $I$ to $R=H^+$ is the identity. 
  Since $IJ=-JI$, it follows that $J$ swaps the subspaces $H^+$ and $H^-$.
  Hence any $\xi\in H_x$ can be uniquely written as $\xi=\ups_1+J\ups_2$, for
  $\ups_1,\ups_2\in R_x$, which is the image of
  $I\otimes\ups_1+J\otimes\ups_2$ under the mapping above.
  Because it is a linear map between the spaces of the same dimension, the
  claim follows.

  (2)
  As above, let us interpret $\xi\in H_x$ as a
  linear map $\R^2\to\R^n$ and $A\in\Cal Q_x$ as an endomorphisms
  $\R^2\to\R^2$.
  Hence $A(\xi)$ is interpreted as the composition $\xi\o A$.
  Easily, $\xi\o A=\xi$ for some $A$ if and only if $A$ is the 
  identity or $\ker\xi$ is non-trivial and invariant under $A$ and there 
  further is a non-zero vector fixed by $A$.
  In other words, the later condition means that $A$ has two real eigenvalues
  1 and $\la$ so that the eigenspace corresponding to $\la$ coincides with
  $\ker\xi$. 
  Considering an endomorphisms $A=aI+bJ+cK$, the eigenvalues of $A$ turn
  out to be the solutions of $\la^2+\det A=0$.
  Hence 1 is an eigenvalue of $A$ if and only if $\det A=-1$. 
  Consequently, the second eigenvalue is $-1$ and $A$ is a skew reflection.
  In terms of split quaternions, $|A|^2=-1$ and $A$ corresponds to a product 
  structure on $H_x$.
  Note that for any one-dimensional subspace $\ell$ in $\R^2$ there is a 
  skew reflection $A=aI+bJ+cK$ whose eigenspace to $-1$ is $\ell$.
  Altogether, since $A$ is never the identity, 
  the equivalence follows.

  (3)
  For the last statement, note that the maximal linear subspaces in the Segre
  cone $\Cal C_x$  have dimension $n$ and all of them are parametrized by
  one-dimensional subspaces in $\R^2$ as follows.
  For any line $\ell\subset\R^2$, let us consider the set $W_\ell$ of all linear maps
  $\xi:\R^2\to \R^n$ so that  $\ker\xi=\ell$.
  Indeed, $W_\ell\subset H_x$ is a linear subspace of dimension $n$ whose
  each  element has got rank one.
  As before, there is a skew reflection $A$ in $\R^2$ so that $A|_\ell=-\id$.
  Hence $\xi\o A=\xi$ for any $\xi\in W_\ell$, i.e. $W_\ell$ is the eigenspace
  to 1 of the corresponding product structure on $H_x$.
  The converse is evident as well.
\end{proof}

For any $A\in\Cal Q_x$, $A\o A$ is a multiple of the identity, hence there is a
natural inner product in each fibre $\Cal Q_x$; the corresponding norm is 
defined by $A\o A=-|A|^2\id$, the signature is $(1,2)$.
Let us denote by $\Cal S$ the line bundle over $M$, whose fiber consists of 
all real multiples of this inner product on $\Cal Q_x$.
For a Lie contact manifold $M$, the corresponding split-quaternionic
structure $\Cal Q$ on $H\subset TM$ has to be compatible with the Levi
bracket, which is expressed by \eqref{eq3} on the level of Lie algebra.
Since $\det A$ corresponds to $|A|^2$ under the above identifications, 
we conclude with the following description of the Lie contact structure.
\begin{prop*}			\label{prop2}
  A Lie contact structure on a smooth manifold $M$ of dimension $2n+1$ 
  is equivalent to the following data:
  \begin{itemize}
  \item a contact distribution $H\subset TM$,
  \item a split-quaternionic structure $\Cal Q=\span{I,J,K}\subset\End H$,
  \item an identification $\Cal S\cong TM/H$,
  \end{itemize}
  such that the Levi bracket is compatible as
  $\L(A(\xi),A(\eta))=|A|^2\.\L(\xi,\eta)$, for any $\xi,\eta\in H_x$ and
  $A\in\Cal Q_x$ in each $x\in M$.
\end{prop*}
\begin{proof}
  According to the previous discussion, we just construct a Lie contact
  structure from the later data.
  By the lemma above, the split-quaternionic structure $\Cal
  Q$ provides an identification $H\cong L^*\otimes R$ where $L$ and $R$ are
  appropriate vector bundles.
  (Namely, $R_x\subset H_x$ is defined as the $+1$-eigenspace of the product
  structure $I$ and $L_x$ as the dual of $\span{I,J}$.)
  Now, for $\ph\in L^*_x$, let us consider the subset 
  $W:=\ph\otimes R_x\subset H_x$. 
  This is a linear $n$-dimensional subspace contained in the Segre cone of 
  $H_x$ and there is  a product structure 
  $A\in\Cal Q_x$ such that $A|_{W}=\id$.
  Hence for any $\xi,\eta\in W$, the compatibility
  $\L(A(\xi),A(\eta))=|A|^2\.\L(\xi,\eta)$ yields $\L(\xi,\eta)=-\L(\xi,\eta)$
  and so $\L(\xi,\eta)=0$.
  In other words, $\ph\otimes R_x$ is an isotropic subspace in $H_x$ for any $\ph$.
  
  According to Proposition \ref{prop1}, in order to determine a Lie contact 
  structure on $M$ we further need to identify the line bundles
  $\La^2L^*$ and $TM/H$ over $M$.
  However, this is provided by the identification $\Cal S\cong TM/H$ assumed above 
  and the identification $\Cal S\cong\La^2L^*$ which follows:
  By definition, $L^*_x$ is identified with $\span{I,J}\subset\Cal Q_x$.
  Hence the correspondence between the volume forms on $\span{I,J}$ and the
  multiples of the natural inner product on $\Cal Q_x$ is obvious.
\end{proof}

\section{Chains}				\label{3}

\subsection{Path geometry of chains}		\label{3.1}
For general parabolic contact geometry $(\G\to M,\om)$ of type $(G,P)$, the
\textit{chains} are defined as projections of flow lines of constant vector 
fields on $\G$ corresponding to non-zero elements of $\g_{-2}$, where
$\g_{-2}$ is the 1-dimensional subspace from the contact grading of $\g$ as
above.
In the homogeneous model $G/P$, the chains are the curves of type 
$t\mapsto g\exp(tX)P$, for $g\in G$ and $X\in\g_{-2}$.
As unparametrized curves, chains are uniquely determined by a tangent direction 
in a point, \cite[section 4]{CSZ}.
By definition, chains are defined only for directions which are transverse to
the contact distribution $H\subset TM$.
In classical terms, the family of chains defines a path geometry on
$M$ restricted, however, only to the directions transverse to $H$.

A \emph{path geometry} on $M$ is a smooth family of unparametrized curves
(paths) on $M$ with the
property that for any point $x\in M$ and any direction $\ell\subset T_xM$ 
there is unique path $C$ from the family such that $T_xC=\ell$.
It turns out that path geometry on $M$ can be equivalently described as a 
parabolic geometry over $\P TM$ of type $(\tG,\tP)$, where
$\tG=PGL(m+1,\R)$, $m=\dim M$, and $\tP$ is the parabolic subgroup as follows.
Let us consider the grading of $\tg=\frak{sl}(m+1,\R)$ which is schematically 
described by the block decomposition with blocks of sizes 1, 1, and $m-1$:
$$
  \pmat{\tg_0&\tg^E_1&\tg_2\\\tg^E_{-1}&\tg_0&\tg^V_1\\\tg_{-2}&\tg^V_{-1}&\tg_0}.
$$
Then $\tp=\tg_0\oplus\tg_1\oplus\tg_2$ is  a parabolic subalgebra of $\tg$ and 
$\tP\subset\tG$ is the subgroup represented by block upper triangular
matrices so that its Lie algebra is $\tp$.
As usual, the parabolic geometry associated to the path geometry on $M$ is uniquely 
determined (up to isomorphism) provided we consider it is regular and 
normal. 

Denote by $\tilde M=\P_0TM$ the open subset of $\P TM$ consisting of all
lines in $TM$ which are transverse to the contact distribution $H\subset TM$.
Now the family of chains on $M$ gives rise to a parabolic geometry
$(\tilde\G\to\tilde M,\tom)$ of type $(\tG,\tP)$ which we call the \textit{path
geometry of chains} and which we are going to describe in a direct way
extending somehow the Cartan geometry $(\G\to M,\om)$.
For this reason it is crucial to observe that $\tilde M$ is naturally
isomorphic to the quotient bundle $\G/Q$, where $Q\subset P$ is the stabilizer
of the 1-dimensional subspace $\g_{-2}\subset\g_-$ under the action  induced from 
the adjoint representation.
Evidently, the Lie algebra of $Q$ is $\q=\g_0\oplus\g_2$.
Hence the couple $(\G\to\tilde M,\om)$ forms a Cartan (but not parabolic)
geometry of type $(G,Q)$.

\subsection{Induced Cartan connection}		\label{3.2}
Starting with the homogeneous model $G/P$ of the parabolic contact geometry, 
let $(\tilde\G\to\widetilde{G/P},\tom)$ be the regular normal parabolic
geometry of type $(\tG,\tP)$ corresponding to the path geometry of chains.
Since $\widetilde{G/P}=\P_0T(G/P)$ is isomorphic to the homogeneous space
$G/Q$, the Cartan geometry $(\tilde\G\to G/Q,\tom)$ is homogeneous under the
action of the group $G$.
Hence it is known there is a pair $(i,\al)$ of maps, consisting of a Lie group
homomorphism $i:Q\to\tP$ 
and a linear map $\al:\g\to\tg$, so that 
$\tilde\G\cong G\x_Q\tP$ and $j^*\tom=\al\o\mu$, where $j$ is the canonical 
inclusion $\G\hookrightarrow\G\x_Q\tP$ and $\mu$ is the Maurer--Cartan form on
$G$.
In particular, the pair $(i,\al)$ has to satisfy the following conditions:
\begin{enumerate}
\item $\al\o\Ad(h)=\Ad(i(h))\o\al$ for all $h\in Q$,
\item the restriction $\al|_\q$ coincides with $i':\q\to\tp$, the derivative of $i$,
\item the map $\underline{\al}:\g/\q\to\tg/\tp$ induced by
  $\al$ is a linear isomorphism.
\end{enumerate}

On the other hand, any pair of maps $(i,\al)$ which are compatible in the above 
sense gives rise to a functor from Cartan geometries of type $(G,Q)$ to 
Cartan geometries of type $(\tG,\tP)$.
There is an easy control over the natural 
equivalence of functors associated to different pairs, see section 3 in \cite{CZ}.

If $\ka$ is the curvature function of the Cartan geometry $(\G\to M,\om)$, 
then the curvature function $\tilde\ka$ of the Cartan geometry induced by
$(i,\al)$ is completely determined by $\ka$ and the contribution of $\al$.
More specifically, this is given by 
\begin{equation}		\label{eq6}
  \tilde\ka\o j=\al\o\ka\o\underline{\al}^{-1} +\Psi_\al,
\end{equation}
where $\Psi_\al:\La^2(\tg/\tp)\to\tg$ is defined as follows.
Consider a bilinear map $\g\x\g\to\tg$ defined by 
\begin{equation*}
  (X,Y)\to [\al(X),\al(Y)]-\al([X,Y]).
\end{equation*}
The map is obviously skew symmetric and, due to the compatibility conditions on
$(i,\al)$, it factorizes to a well-defined $Q$-equivalent map $\g/\q\x\g/\q\to\tg$.
The map $\Psi_\al$ is then obtained according to  the isomorphism 
$\underline{\al}^{-1}:\tg/\tp\to\g/\q$.
Note that $\Psi_\al$ vanishes if and only if $\al$ is a homomorphism of Lie algebras.

\subsection{The pair}				\label{3.4}
Here we describe the pair $(i,\al)$ for Lie contact structures explicitly.
Considering the dimension of the base manifold $M$ is $m=2n+1$, we have to consider 
$G=PO(p+2,q+2)$, $p+q=n$, and $\tG=PGL(2n+2,\R)$.
According to the definition in \ref{3.1}, the subgroup $Q\subset P$ 
is represented by the block matrices of the form
\begin{equation*}
  \pmat{B&0&wC\Bbb J\\ 0&C&0\\ 0&0&(B^{-1})^t},
\end{equation*}
where $\Bbb J=\pmat{0&1\\-1&0}$, $w\in\R$, $B\in GL(2,\R)$, and $C\in O(p,q)$.
If we denote $\be:=\det B$ and substitute $B=\pmat{p&r\\s&q}$, then the two maps 
$i:Q\to\tP$ and $\al:\g\to\tg$ are defined explicitly by
\begin{eqnarray*}
i\pmat{
p&r&0&-rw&pw\\ 
s&q&0&-qw&sw\\ 
0&0&C&0&0\\ 
0&0&0&\frac q\be&-\frac s\be\\ 
0&0&0&-\frac r\be&\frac p\be} 
&:=& \pmat{
\sqrt{|\be|}&-w\sqrt{|\be|}&0&0\\ 
0&\frac1{\sqrt{|\be|}}&0&0\\ 
0&0&\frac q{\sqrt{|\be|}} C&-\frac s{\sqrt{|\be|}} C\\ 
0&0&-\frac r{\sqrt{|\be|}} C&\frac p{\sqrt{|\be|}} C},
\end{eqnarray*}
\begin{eqnarray*}
\al\pmat{
a&b&U_1&0&w\\ 
c&d&U_2&-w&0\\ 
X_1&X_2&D&\I_{p,q} U_1^t&\I_{p,q} U_2^t\\ 
0&z&X_1^t\I_{p,q}&-a&-c\\
-z&0&X_2^t\I_{p,q}&-b&-d} 
&:=& \pmat{
\frac{a+d}2&-w&\frac12U&\frac12V\\
z&-\frac{a+d}2&-\frac12Y^t\I_{p,q}&\frac12X^t\I_{p,q}\\ 
X_1&-\I_{p,q} U_2^t&D+\frac{d-a}2\I_n&-c\I_n\\
X_2&\I_{p,q} U_1^t&-b\I_n&D+\frac{a-d}2\I_n}.
\end{eqnarray*}

\begin{prop*}			\label{lem1}
  The map $i:Q\to\tilde P$ is an injective Lie group homomorphism, 
  $\al:\g\to\tg$ is linear, and the pair $(i,\al)$ satisfies the conditions 
  (1)--(3) from \ref{3.2}. 
  \end{prop*}
  Hence the pair $(i,\al)$ gives rise to an extension functor from Cartan
  geometries of type $(G,Q)$ to Cartan geometries of type $(\tilde G,\tilde P)$.
\begin{proof}
  The map $i$ 
  is obviously well defined, i.e.\ the image of an element of $Q\subset PO(p+2,q+2)$ 
  does not depend on the representative matrix.
  Further, $i$ is smooth and injective and a direct computation shows this is a 
  homomorphism of Lie groups.
  The map $\al$ is linear and the compatibility of the pair
  $(i,\al)$ follows as follows:
  (1) involves a little tedious but very straightforward checking,
  (2) is an easy exercise,
  (3) follows from $\g/\q\cong\g_-\oplus\g_1$ and $\tg_-\cong\tg/\tp$,
  restricting $\al$ to $\g_-\oplus\g_1$.
\end{proof}

\subsection{The properties}	\label{3.5}
Here we analyze the map $\Ps_\al:\La^2(\tg/\tp)\to\tg$ from \ref{3.2}.
For this reason we need a bit of notation:

As a linear space, $\tg/\tp$ can be identified with
$\tg_-=\tg_{-1}^E\oplus\tg_{-1}^V\oplus\tg_{-2}$. 
Using brackets in $\tg$, we may identify $\tg_{-1}^V$ with
$\tg_1^E\otimes\tg_{-2}$ if necessary. 
We will further view $\tg_{-2}$ as $\R^{2n}=\R^n\x\R^n$ and correspondingly 
write $X\in\tg_{-2}$ as $\pmat{X_1\\X_2}$ for $X_1,X_2\in\R^n$. 
Next, the semi-simple part $\tg_0^{ss}$ of $\tg_0$ is isomorphic to
$\frak{sl}(2n,\R)$ and the restriction of the adjoint representation of
$\tg_0^{ss}$ on $\tg_{-2}$ coincides with the standard representation of 
$\frak{sl}(2n,\R)$ on $\R^{2n}$.
Finally, by $\span{\ ,\ }$ we denote the standard inner product of signature
$(p,q)$ on $\R^n$ as above, i.e.\ $\span{X_1,X_2}=X_1^t\I_{p,q}X_2$ for any
$X_1,X_2\in\R^n$.

\begin{lem*}			\label{lem2}
  Viewing the map $\Ps_\al$ from \ref{3.2} as an element of 
  $(\tg_-)^*\wedge(\tg_-)^*\otimes\tg$,  then
  it lies in the subspace $(\tg_{-1}^V)^*\wedge(\tg_{-2})^*\otimes\tg_0^{ss}$. 
  Denoting by $W_0$ a non-zero element of $\tg_1^E$, then the trilinear map 
  $\tg_{-2}\x \tg_{-2}\x \tg_{-2}\to\tg_{-2}$
  defined by $(X,Y,Z)\mapsto [\Ps_{\al}(X,[Y,W_0]),Z]$ is (up to a non-zero
  multiple) the complete symmetrization of the map 
  $$
  (X,Y,Z)\mapsto \pmat{\span{X_1,Y_2}Z_1-\span{X_1,Y_1}Z_2\\ 
  \span{X_2,Y_2}Z_1-\span{X_1,Y_2}Z_2}.
  $$  
\end{lem*}
\begin{proof}
  For $X,Y\in\tg_-$, let $\hat X,\hat Y$ be the unique
  elements in $\g_-\oplus\g_1$ so that $\al(\hat X)$ and $\al(\hat Y)$ is congruent to
  $X$ and $Y$ modulo $\tp$, respectively.
  By the definition, one computes directly 
  $\Ps_\al(X,Y)=[\al(\hat X),\al(\hat Y)]-\al([\hat X,\hat Y])$ which has always
  values in the lower right $2n\x2n$ block of $\tg$ with vanishing trace, i.e.\
  in the semi-simple part of $\tg_0$.
  Further, $\Ps_\al$ vanishes whenever both the entries $X,Y$ lie 
  in $\tg^V_{-1}$ or in $\tg_{-2}$ or any of them lies in $\tg^E_{-1}$.
  Hence the map $\Ps_\al$ is indeed of the form $\tg^V_{-1}\wedge\tg_{-2}\to\g_0^{ss}$.
  
  Considering $X,Y\in\tg_{-2}$, the element $[Y,W_0]$ lies in $\tg^V_{-1}$ and the 
  non-zero $2n\x 2n$ block of $\Ps_\al(X,[Y,W_0])$  looks explicitly like
  $$
  \pmat{\frac12(R_{12}+\tr R_{12}\.\I_n)-R_{21} &\frac12(R_{11}-\tr R_{11}\.\I_n) \\
  	\frac12(R_{22}-\tr R_{22}\.\I_n) & R_{12}-\frac12(R_{21}+\tr R_{21}\.\I_n) },
  $$
  where $R_{11}=(X_1Y_1^t+Y_1X_1^t)\I_{p,q}$,
  $R_{22}=(X_2Y_2^t+Y_2X_2^t)\I_{p,q}$, 
  $R_{12}=(X_1Y_2^t+Y_1X_2^t)\I_{p,q}$, and $R_{21}=R_{12}^t$.
  The value of $[\Ps_{\al}(X,[Y,W_0]),Z]$ is then obtained by applying the above
  matrix to the vector $Z=\pmat{Z_1\\Z_2}$ from $\tg_{-2}=\R^n\x\R^n$.
  The result turns out to be the cyclic sum of
  $$
  \frac12(\span{Y_1,X_2}+\span{Y_2,X_1})\pmat{Z_1\\-Z_2}
  +\pmat{-\span{X_1,Y_1}Z_2\\\span{X_2,Y_2}Z_1},
  $$
  which is  the complete symmetrization of 
  $$
  \span{X_1,Y_2}\pmat{Z_1\\-Z_2}+\pmat{-\span{X_1,Y_1}Z_2\\\span{X_2,Y_2}Z_1}
  $$
  up to a non-zero multiple.
\end{proof}

In \ref{3.2} we motivated the definition of the pair $(i,\al)$.
In order to justify the choice we made in \ref{3.4}, we have to
check that  starting with the homogeneous model, the
associated Cartan geometry determined by $(i,\al)$ is regular and normal,
i.e.\ it is the canonical Cartan geometry describing the path geometry of
chains on $G/P$.
This is provided by the following Theorem.
\begin{thm}			\label{th2}
  Let $\dim M\ge 7$ and $(\Cal G\to M,\om)$ be a regular normal parabolic geometry of
  type $(G,P)$ and let $(\tilde\G:=\Cal G\x_Q\tilde P\to\P_0TM,\tilde\om_\al)$
  be the parabolic geometry obtained using the
  extension functor associated to the pair $(i,\al)$ defined in \ref{3.4}. 
  Then this geometry is regular and normal if and only if $(\Cal G\to M,\om)$ is locally flat. 
  In that case the induced Cartan geometry is non-flat and torsion free.
\end{thm}
\begin{proof}
  If $(\Cal G\to M,\om)$ is locally flat parabolic geometry of type $(G,P)$ then
  by \eqref{eq6} the curvature of the induced Cartan geometry is determined only by
  $\Ps_\al$.
  Since $\Ps_\al$ in non-trivial, the induced Cartan geometry in non-flat and
  since the values of $\Ps_\al$ are in $\tg_0\subset\tp$, it  is torsion free 
  and hence regular.
  In order to prove the normality, one has to show that $\del^*\Ps_\al=0$.
  This can be proved either by a direct computation or by a more conceptual 
  argument as   in the proof of Theorem 3.6 in \cite{CZ}.

  Conversely, let us suppose the induced Cartan geometry 
  $(\tilde\G\to\P_0TM,\tom_\al)$ is regular and normal.
  Then, analyzing the harmonic curvature of $\tom_\al$ as in the first part of the
  proof of Theorem 3.8 in \cite{CZ}, it easily follows the Cartan geometry $(\Cal G\to
  M,\om)$ is necessarily torsion free.
  If $\dim M\ge 7$, different from 9, there are two harmonic curvature
  components of $\om$; if $\dim M=9$, one of the two components above is
  further split into two parts. 
  In any case, all the components are of homogeneity one,
  i.e.\ all are torsions.
  Hence for $\dim M\ge 7$, torsion free  and normal parabolic geometry of
  type $(G,P)$ is flat, i.e.\ locally isomorphic to the homogeneous model.
\end{proof}

\section{Applications}		\label{4}

The applications in this section are based on the normality of the induced
Cartan geometry. 
Hence according to Theorem \ref{th2} we consider the Lie contact
structure is locally flat.

\begin{thm}			\label{th3}
  Let $M$ be a contact manifold of dimension $\ge 7$ with locally flat 
  Lie contact structure.
  Then there is no linear connection on  $TM$ which has the chains among 
  its geodesics. 
\end{thm}
\begin{proof}
  Since we deal with locally flat Lie contact structures, the induced path
  geometry of chains is regular and normal by Theorem \ref{th2} and its
  curvature is completely determined by the map $\Ps_\al$.
  From lemma \ref{lem2}  we know that $\Ps_\al$ is of homogeneity three, hence
  it must coincide with the unique harmonic curvature component
  which is there in this  homogeneity for generalized path geometries, 
  see e.g.\ the summary in section 3.7 in \cite{CZ}.
  However by the second part of Theorem 4.7 in \cite{C1}, the vanishing of
  this component is equivalent to the fact the path geometry comes from a
  projective structure on $M$.
  Since $\Ps_\al\ne 0$, the claim follows.
\end{proof}

In other words, even working with the locally flat Lie contact structures, the
family of chains forms a rather complicated system of curves. 
In terms of second order ODE's, the chain equation is
never locally equivalent either to the trivial equation or to the geodesic
equation.
In particular, considering the homogeneous model, 
the chain equation provides an example of non-trivial torsion free second
order ODE  with a reasonably large automorphism group.
More specifically, the automorphism group obviously contains $G=PO(p+2,q+2)$,
i.e.\ it has dimension at least $\frac{n^2+5n+6}2$, provided $n=p+q$ as before.
We will see in the next section the dimension 
actually equals to the dimension of $G$.

\subsection{The reconstruction}
As we know from \ref{3.5}, the harmonic curvature $\tilde\ka_H$ of the induced Cartan
geometry is a section of the bundle associated to
$(\tg_{-1}^V)^*\otimes(\tg_{-2})^*\otimes\tg_0^{ss}$. 
We are going to interpret this quantity geometrically, which will allow us to
reconstruct the Lie contact structure from the path geometry of chains.

Let $(\G\to M,\om)$ be the parabolic geometry corresponding to a locally 
flat Lie contact structure on $M$.
Let $\tilde M=\P_0TM$ and let $(\tilde\G\to\tilde M,\tom)$ be the parabolic
geometry induced by the path geometry of chains.
Let us denote by $\tilde E$ and $\tilde V$ the subbundles in $T\tilde M$ 
corresponding to the subspaces $\tg_{-1}^E$  and $\tg_{-1}^V$ in $\tg_{-1}$,
respectively. 
Let us further denote $\tilde F:=T\tilde M/(\tilde E\oplus\tilde V)$; as an
associated bundle over $\tilde M$ this corresponds to
$\tg_{-2}\cong\tg_-/\tg_{-1}$. 
As before, we can replace the space 
$(\tg_{-1}^V)^*\otimes(\tg_{-2})^*\otimes\tg_0^{ss}$ by
$\tg_{-1}^E\otimes(\otimes^3(\tg_{-2})^*)\otimes\tg_{-2}$;
the corresponding associated bundle is  
$\tilde E\otimes(\otimes^3\tilde F^*)\otimes\tilde F$. 
Altogether, since $\tilde E\subset T\tilde M$ is a line bundle, we
can view the harmonic curvature $\tilde\ka_H$ as a section of the bundle
$\otimes^3\tilde F^*\otimes\tilde F\to\tilde M$ determined up to a non-zero multiple.
(More specifically, lemma \ref{lem2} shows it is actually completely
symmetric and trace free.)

In order to express $\tilde\ka_H$ in terms of the underlying Lie contact structure  
on $M$, let us fix $x\in M$ and $\ell\in\pi^{-1}(x)$, where $\pi:\tilde M\to M$
is the natural projection. 
(By definition, $\ell$ is a line in $T_xM$ which is transverse to the contact
distribution $H_x\subset T_xM$.)
For each $\xi\in T_xM$, let $\tilde\xi\in T_\ell\tilde M$ be any lift and consider 
its class  in $\tilde F_\ell=T_\ell\tilde M/(\tilde
E_\ell\oplus\tilde V_\ell)$. 
Since $\tilde V\subset T\tilde M$ is the vertical subbundle of the projection $\pi$, 
this class is independent of the choice of the lift.
From the explicit description of the tangent map of the projection $\pi$ it easily 
follows that its appropriate restriction yields a linear isomorphism 
$H_x\cong\tilde F_\ell$.
Altogether, for a fixed $x$ and $\ell$, the $\tilde\ka_H$ gives rise to an element 
of $\otimes^3H_x^*\otimes H_x$, denoted by $S_x$, which is determined up to a 
non-zero multiple.
Its explicit description and the geometrical meaning are as follows.
\begin{lem*}				\label{lem3}
  Let $(M,H=L^*\otimes R)$ be a locally flat Lie contact manifold and let
  $\Cal Q=\span{I,J,K}$ be the corresponding split-quaternionic structure on
  $H$ as in \ref{2.4}. 
  Let $S_x$ be the element of $\otimes^3H_x^*\otimes H_x$  constructed from 
  the harmonic curvature of the associated path geometry of chains as above.

  (1) Then, up to a non-zero multiple, $S$ is the cyclic  sum of the mapping
  $$
    (\xi,\eta,\ze)\mapsto\L(\xi,I\eta)I\ze+\L(\xi,J\eta)J\ze-\L(\xi,K\eta)K\ze,
  $$
  which is independent of the choice of basis of $\Cal Q$. 

  (2) A non-zero element $\xi$ in $H_x=L_x^*\otimes R_x$ is of rank one if and only if 
  $S(\xi,\xi,\xi)=0$ or there is $\eta\in H_x$ such that $S(\xi,\xi,\eta)$ is
  a non-zero multiple $\xi$.
\end{lem*}
Although the Levi bracket $\L$ has values in $TM/H$, we view it here as a
real valued bilinear form, because of the non-zero multiple freedom.
\begin{proof}
  (1) Since $\tilde\ka_H$ is given by $\Psi_\al$, we just need to compare the
  map above with the one described in lemma \ref{lem2}.
  On the level of Lie algebra $\g$, the Levi bracket $\L$ corresponds to $[\
  ,\ ]:\g_{-1}\x\g_{-1}\to\g_{-2}$, which is explicitly described by
  \eqref{eq1} in \ref{2.1}.
  The operators $I,J$, and $K$ correspond to the multiplication by
  matrices \eqref{eq5} on $\g_{-1}=\R^{2*}\otimes\R^n$ from the right.
  Representing $\xi\in H_x$ by a matrix  $(X_1,X_2)$ with columns 
  $X_1,X_2\in\R^n$,  the images $I\xi$, $J\xi$ and $K\xi$ corresponds to 
  $(X_1,-X_2)$, $(X_2,X_1)$, and $(-X_2,X_1)$, respectively.
  Representing also $\eta$ and $\ze$ by $(Y_1,Y_2)$ and $(Z_1,Z_2)$, 
  the expressions $\L(\xi,I\eta)$, $\L(\xi,J\eta)$, and $\L(\xi,K\eta)$, 
  correspond then to 
  $-\span{X_1,Y_2}-\span{X_2,Y_1}$, $\span{X_1,Y_1}-\span{X_2,Y_2}$, and 
  $\span{X_1,Y_1}+\span{X_2,Y_2}$, respectively.
  Hence the direct substitution yields that the two mappings correspond each 
  other up to a non-zero multiple.
  The independence of the choice of basis of $\Cal Q$ follows by a
  straightforward checking.

  (2) In the above terms, $\xi$ is of rank one if and only if the
  corresponding vectors $X_1$ and $X_2$ are linearly dependent.
  Using this it is then easy to check that $S(\xi,\xi,\xi)=0$ and, moreover,
  that $S(\xi,\xi,\eta)$ equals to a multiple of $\xi$, for any $\eta$.

  For the converse statement, let us distinguish the two cases:
  Firstly, suppose $\L(\xi,I\xi)$, $\L(\xi,J\xi)$, and $\L(\xi,K\xi)$ does 
  not vanish simultaneously.
  Since the assumption $S(\xi,\xi,\xi)=0$ is equivalent to
  $\L(\xi,I\xi)I\xi+\L(\xi,J\xi)J\xi-\L(\xi,K\xi)K\xi=0$, it follows that 
  if two of the summands vanish then $\xi$ will be 0.
  Hence if $\xi\ne 0$ and $S(\xi,\xi,\xi)=0$ then at most one of the three
  summands vanishes and this yields that $\xi=A\xi$ for a specific element 
  $A\in\span{I,J,K}$.
  Hence $\xi$ has got rank one by lemma \ref{lem0}.

  Secondly, suppose that $\L(\xi,I\xi)=\L(\xi,J\xi)=\L(\xi,K\xi)=0$.
  Then $S(\xi,\xi,\xi)$ vanishes trivially, however, it also turns out that 
  $S(\xi,\xi,\eta)$ is a non-zero multiple of 
  $\L(\xi,I\eta)I\xi+\L(\xi,J\eta)J\xi-\L(\xi,K\eta)K\xi$, for any $\eta$.
  Hence the assumption there is $\eta$ so that $S(\xi,\xi,\eta)=\xi$ yields that 
  there is $A\in\span{I,J,K}$ so that $A\xi=\xi$.
  The rest follows again by lemma \ref{lem0}.
\end{proof}

According to the development in \ref{2.3}, to recover the Lie contact 
structure on $M$ it is enough to determine the rank-one elements in $H_x$.
The above lemma provides this in terms of $S_x$, hence we conclude by the
following interesting result:
\begin{thm*}				\label{th4}
  Let $M$ be a locally flat Lie contact manifold.
  Then the Lie contact structure can be reconstructed from the harmonic curvature 
  of the regular normal
  parabolic geometry associated to the path geometry of chains.
  Consequently, a contact diffeomorphism on $M$ which maps chains to chains
  is an automorphism of the Lie contact structure.
\end{thm*}
\begin{proof}
  A contact diffeomorphism $f$ on $M$ lifts to a diffeomorphism $\tilde f$ on 
  $\tilde M=\Cal P_0TM$.
  The assumption that $f$ preserves chains yields that $\tilde f$ is an
  automorphism of the associated path geometry of chains.
  In particular, $\tilde f$ is compatible with the harmonic curvature
  $\tilde\ka_H$, which by assumption corresponds to the mapping $S$ above.
  Hence $f$ is compatible with $S$ and the rest follows.
\end{proof}

\subsection{Final remarks}
As we noted in remark \ref{rem1}, the Lie contact structures in dimension 5
and 3 are basically equivalent to other parabolic contact structures. 
The procedure of previous sections is in effect  independent of
the dimension, however, the difference in these cases is that 
there are harmonic curvature components with higher homogeneity. 
Hence the torsion freeness condition does not imply the local flatness as in
\ref{th2} and the results are more general.
Consult the corresponding sections in \cite{CZ} and \cite{CZ2} for details.


\end{document}